\documentclass[11pt]{article}
\usepackage{amsmath}
\usepackage{amsthm}
\usepackage{graphicx}
\usepackage{amscd}
\usepackage{float}
\usepackage{subfigure}
\usepackage{algpseudocode}
\usepackage{algorithm}%
\usepackage{amsfonts}%
\usepackage{amssymb}
\newtheorem{theorem}{{\bf Theorem}}

\newtheorem{definition}{{\bf Definition}}

\newtheorem{lemma}{{\bf Lemma}}
\newtheorem{remark}{{\bf Remark}}

\begin{document}

\title{Sharp exponential inequalities in survey sampling:\\ conditional Poisson sampling schemes}
\author{Patrice Bertail\\
MODALX - Universit\'{e} Paris Ouest Nanterre\\
200 avenue de la R\'epublique, Nanterre 92100, France\\
patrice.bertail@u-paris10.fr\\
\and Stephan Cl\'{e}men\c{c}on \\
LTCI, Telecom ParisTech, Universit\'e Paris-Saclay\\ 46 rue Barrault, 75013, Paris, France\\
stephan.clemencon@telecom-paristech.fr}
\maketitle

\begin{abstract}%

This paper is devoted to establishing exponential bounds for the probabilities of deviation of a sample sum from its expectation, when the variables involved in the summation are obtained by sampling in a finite population according to a rejective scheme, generalizing sampling without replacement, and by using an appropriate normalization. In contrast to Poisson sampling, classical deviation inequalities in the i.i.d. setting do not straightforwardly apply to sample sums related to rejective schemes, due to the inherent dependence structure of the sampled points. We show here how to overcome this difficulty, by combining the formulation of rejective sampling as Poisson sampling conditioned upon the sample size with the Escher transformation. In particular, the Bennett/Bernstein type bounds established highlight the effect of the asymptotic variance $\sigma^2_{N}$ of the (properly standardized) sample weighted sum and are shown to be much more accurate than those based on the negative association property shared by the terms involved in the summation. Beyond its interest in itself, such a result for rejective sampling is crucial, insofar as it can be extended to many other sampling schemes, namely those that can be accurately approximated by rejective plans in the sense of the total variation distance.
\\
\textbf{AMS 2015 subject classification:} 60E15, 6205.\newline
\textbf{Keywords and phrases:} Exponential inequality ; Poisson survey scheme; rejective sampling; survey sampling; Escher transformation; coupling
\end{abstract}%

\section{Introduction}

Whereas many upper bounds for the probability that a sum of independent
real-valued (integrable) random variables exceeds its expectation by a
specified threshold value $t\in\mathbb{R}$ are documented in the literature
(see \textit{e.g.} \cite{BLM13} and the references therein), very few results
are available when the random variables involved in the summation are sampled
from a finite population according to a given survey scheme and next
appropriately normalized (using the related survey weights as originally
proposed in \cite{HT51} for approximating a total). The sole situation where
results in the independent setting straightforwardly carry over to survey
samples (without replacement) corresponds to the case where the variables are
sampled independently with possibly unequal weights, \textit{i.e.} Poisson
sampling. For more complex sampling plans, the dependence structure between
the sampled variables makes the study of the fluctuations of the resulting
weighted sum approximating the total (referred to as the
\textit{Horvitz-Thompson total estimate}) very challenging. The case of basic
\textit{sampling without replacement} (\textsc{SWOR} in abbreviated form) has
been first considered in \cite{Hoeffding63}, and refined in \cite{Serfling74}
and \cite{BardenetMaillard}. In contrast, the asymptotic behavior of the
Horvitz-Thompson estimator as $N$ tends to infinity is well-documented in the
litterature. Following in the footsteps of the seminal contribution
\cite{Hajek64}, a variety of limit results (\textit{e.g.} consistency,
asymptotic normality) have been established for Poisson sampling and next
extended to rejective sampling viewed as conditional Poisson sampling given
the sample size and to sampling schemes that are closed to the latter in a
\textit{coupling} sense in \cite{Rob82} and \cite{Ber98}. Although the nature
of the results established in this paper are nonasymptotic, these arguments
(conditioning upon the sampling size and coupling) are involved in their proofs.

It is indeed the major purpose of this article to extend tail bounds proved
for \textsc{SWOR} to the case of rejective sampling, a fixed size sampling
scheme generalizing it. The approach we develop is thus based on viewing
rejective sampling as conditional Poisson sampling given the sample size and
writing then the deviation probability as a ratio of two quantities: the joint
probability that a Poisson sampling-based total estimate exceeds the threshold
$t$ and the size of the cardinality of the Poisson sample equals the
(deterministic) size $n$ of the rejective plan considered in the numerator and
the probability that the Poisson sample size is equal to $n$ in the
denominator. Whereas a sharp lower bound for the denominator can be
straightforwardly derived from a local Berry-Esseen bound proved in
\cite{Deheuvels} for sums of independent, possibly non indentically
distributed, Bernoulli variables, an accurate upper bound for the numerator can be
established by means of an appropriate exponential change of measure
(\textit{i.e.} Escher transformation), following in the footsteps of the
method proposed in \cite{Talagrand95}, a refinement of the classical argument
of Bahadur-Rao's theorem in order to improve exponential bounds in the
independent setting. The tail bounds (of Bennett/Bernstein type) established
by means of this method are shown to be sharp in the sense that they
explicitely involve the 'small' asymptotic variance of the Horvitz-Thompson total
estimate based on rejective sampling, in contrast to those proved by using the
\textit{negative association} property of the sampling scheme.

The article is organized as follows. A few key concepts pertaining to survey
theory are recalled in section \ref{sec:background}, as well as specific
properties of Poisson and rejective sampling schemes. For comparison purpose,
preliminary tail bounds in the (conditional) Poisson case are stated in
section \ref{sec:Poisson}. The main results of the paper, sharper exponential
bounds for conditional Poisson sampling namely, are proved in section
\ref{sec:main}, while section \ref{sec:extension} explains how they can be
extended to other sampling schemes, sufficiently close to rejective sampling in the sense of the total variation norm. A few remarks are finally collected in section \ref{sec:concl} and some
technical details are deferred to the Appendix section.

\section{Background and Preliminaries}

\label{sec:background} As a first go, we start with briefly recalling basic
notions in survey theory, together with key properties of (conditional)
Poisson sampling schemes. Here and throughout, the indicator function of any
event $\mathcal{E}$ is denoted by $\mathbb{I}\{\mathcal{E}\}$, the power set
of any set $E$ by $\mathcal{P}(E)$, the variance of any square integrable r.v.
$Y$ by $Var(Y)$, the cardinality of any finite set $E$ by $\#E$ and the Dirac
mass at any point $a$ by $\delta_{a}$. For any real number $x$, we set
$x^{+}=\max\{x,\; 0 \}$, $x^{-}=\max\{ -x,\; 0 \}$, $\lceil x\rceil=\inf\{
k\in\mathbb{Z}:\; x\leq k \}$ and $\lfloor x \rfloor=\sup\{ k\in\mathbb{Z}:\;
k\leq x \}$.

\subsection{Sampling schemes and Horvitz-Thompson estimation}

Consider a finite population of $N\geq1$ distinct units, $\mathcal{I}%
_{N}=\{1,\ \ldots,\; N \}$ say, a survey sample of (possibly random) size
$n\leq N$ is any subset $s=\{i_{1},\; \ldots,\; i_{n(s)} \}\in\mathcal{P}%
(\mathcal{I}_{N})$ of size $n(s)=n$. A sampling design without replacement is
defined as a probability distribution $R_{N}$ on the set of all possible
samples $s\in\mathcal{P}(\mathcal{I}_{N})$. For all $i\in\mathcal{I}_{N}$, the
probability that the unit $i$ belongs to a random sample $S$ defined on a
probability space $(\Omega,\; \mathcal{F},\; \mathcal{P})$ and drawn from
distribution $R_{N}$ is denoted by $\pi_{i}=\mathbb{P}\{i\in S \}=R_{N}%
(\{i\})$. The $\pi_{i}$'s are referred to as \textit{first order inclusion
probabilities}. The \textit{second order inclusion probability} related to any
pair $(i,j)\in\mathcal{I}_{N}^{2}$ is denoted by $\pi_{i,j}=\mathbb{P}%
\{(i,j)\in S^{2} \}=R_{N}(\{i,\; j \})$ (observe that $\pi_{i,i}=\pi_{i}$).
Here and throughout, we denote by $\mathbb{E}[.]$ the $\mathbb{P}$-expectation
and by $Var(Z)$ the conditional variance of any $\mathbb{P}$-square integrable
r.v. $Z:\Omega\rightarrow\mathbb{R}$.

The random vector $\boldsymbol{\epsilon} _{N}=(\epsilon_{1},\; \ldots,\;
\epsilon_{N})$ defined on $(\Omega,\; \mathcal{F},\; \mathcal{P})$, where
$\epsilon_{i}=\mathbb{I}\{ i\in S \}$ fully characterizes the random sample
$S\in\mathcal{P}(\mathcal{I}_{N})$. In particular, the sample size $n(S)$ is
given by $n=\sum_{i=1}^{N}\epsilon_{i}$, its expectation and variance by
$\mathbb{E}[n(S)]=\sum_{i=1}^{N}\pi_{i}$ and $Var(n(S))=\sum_{1\leq i,\; j\leq
N}\{\pi_{i,j}-\pi_{i}\pi_{j}\}$ respectively. The $1$-dimensional marginal
distributions of the random vector $\boldsymbol{\epsilon} _{N}$ are the
Bernoulli distributions $Ber(\pi_{i})=\pi_{i}\delta_{1}+(1-\pi_{i})\delta_{0}%
$, $1\leq i \leq N$ and its covariance matrix is $\Gamma_{N}=(\pi_{i,j}%
-\pi_{i}\pi_{j})_{1\leq i,\;j \leq N}$. \medskip

We place ourselves here in the \textit{fixed-population} or
\textit{design-based} sampling framework, meaning that we suppose that a fixed
(unknown) real value $x_{i}$ is assigned to each unit $i\in\mathcal{I}_{N}$.
As originally proposed in the seminal contribution \cite{HT51}, the
Horvitz-Thompson estimate of the population total $S_{N}=\sum_{i=1}^{N} x_{i}$
is given by
\begin{equation}
\label{eq:HT}\widehat{S}_{\boldsymbol{\pi} _{N}}^{\boldsymbol{\epsilon} _{N}%
}=\sum_{i=1}^{N} \frac{\epsilon_{i}}{\pi_{i}}x_{i}=\sum_{i\in S}\frac{1}%
{\pi_{i}}x_{i},
\end{equation}
with $0/0=0$ by convention. Throughout the article, we assume that the
$\pi_{i}$'s are all strictly positive. Hence, the conditional expectation of
\eqref{eq:HT}  is $\mathbb{E}[\widehat{S}_{\boldsymbol{\pi} _{N}%
}^{\boldsymbol{\epsilon} _{N}}]=S_{N}$ and, in the case where the size of the
random sample is deterministic, its variance is given by
\begin{equation}
Var(\widehat{S}_{\boldsymbol{\pi} }^{\boldsymbol{\epsilon} _{N}})=\sum_{i<
j}\left(  \frac{x_{i}}{\pi_{i}}-\frac{x_{j}}{\pi_{j}} \right)  ^{2}%
\times\left(  \pi_{i}\pi_{j}-\pi_{i,j} \right)  .
\end{equation}

The goal of this paper is to establish accurate bounds for tail probabilities
\begin{equation}
\label{eq:tailprob}\mathbb{P}\{\widehat{S}_{\boldsymbol{\pi} _{N}%
}^{\boldsymbol{\epsilon} _{N}}-S_{N} >t \},
\end{equation}
where $t\in\mathbb{R}$, when the sampling scheme $\boldsymbol{\epsilon} _{N}$
is \textit{rejective}, a very popular sampling plan that generalizes \textit{random
sampling without replacement} and can be expressed as a conditional Poisson
scheme, as recalled in the following subsection for clarity. One may refer to
\cite{Dev87} for instance for an excellent account of survey theory, including
many more examples of sampling designs.

\subsection{Poisson and conditional Poisson sampling}

\label{subsec:Poisson} Undoubtedly, one of the simplest sampling plan is the
\textit{Poisson survey scheme} (without replacement), a generalization of
\textit{Bernoulli sampling} originally proposed in \cite{Goodman} for the case
of unequal weights: the $\epsilon_{i}$'s are independent and the sampling
distribution $P_{N}$ is thus entirely determined by the first order inclusion
probabilities $\mathbf{p}_{N}=(p_{1},\;\ldots,\;p_{N})\in]0,1[^{N}$:
\begin{equation}
\forall s\in\mathcal{P}(\mathcal{I}_{N}),\;\;P_{N}(s)=\prod_{i\in S}p_{i}%
\prod_{i\notin S}(1-p_{i}). \label{eq:Poisson}%
\end{equation}
Observe in addition that the behavior of the quantity \eqref{eq:HT}  can be
investigated by means of results established for sums of independent random
variables. However, the major drawback of this sampling plan lies in the
random nature of the corresponding sample size, impacting significantly the
variability of \eqref{eq:HT}. The variance of the Poisson sample size is
given by $d_{N}=\sum_{i=1}^{N}p_{i}(1-p_{i})$, while the variance of
\eqref{eq:HT}  is in this case:
\[
Var\left(  \widehat{S}_{\boldsymbol{\pi} _{N}}^{\boldsymbol{\epsilon} _{N}%
}\right)  =\sum_{i=1}^{N}\frac{1-p_{i}}{p_{i}}x_{i}^{2}.
\]
For this reason, \textit{rejective sampling}, a sampling design $R_{N}$ of
fixed size $n\leq N$, is often preferred in practice. It generalizes the
\textit{simple random sampling without replacement} (where all samples with
cardinality $n$ are equally likely to be chosen, with probability $(N-n)!/n!$,
all the corresponding first and second order probabilities being thus equal to
$n/N$ and $n(n-1)/(N(N-1))$ respectively). Denoting by $\boldsymbol{\pi}
_{N}^{R}=(\pi_{1}^{R},\;\ldots,\;\pi_{N})$ its first order inclusion
probabilities and by $\mathcal{S}_{n}=\{s\in\mathcal{P}(\mathcal{I}%
_{N}):\;\#s=n\}$ the subset of all possible samples of size $n$, it is defined
by:
\begin{equation}
\forall s\in\mathcal{S}_{n},\;\;R_{N}(s)=C\prod_{i\in s}p_{i}^{R}%
\prod_{i\notin s}(1-p_{i}^{R}), \label{eq:Rejective}%
\end{equation}
where $C=1/\sum_{s\in\mathcal{S}_{n}}\prod_{i\in s}p_{i}^{R}\prod_{i\notin
s}(1-p_{i}^{R})$ and the vector of parameters $\mathbf{p}_{N}^{R}=(p_{1}^{R},\;\ldots
,\;p_{N}^{R})\in]0,1[^{N}$ yields first order inclusion probabilities equal to
the $\pi_{i}^{R}$'s and is such that $\sum_{i=1}^{N}p_{i}^{R}=n$. Under this
latter additional condition, such a vector $\mathbf{p}_{N}^{R}$ exists and is
unique (see \cite{Dupacova}) and the related representation
\eqref{eq:Rejective}  is then said to be \textit{canonical}. Notice
incidentally that any vector $\mathbf{p}_{N}^{\prime}\in]0,1[^{N}$ such that
$p_{i}^{R}/(1-p_{i}^{R})=cp_{i}^{\prime}/(1-p_{i}^{\prime})$ for all
$i\in\{1,\;\ldots,\;n\}$ for some constant $c>0$ can be used to write a
representation of $R_{N}$ of the same type as \eqref{eq:Rejective}. Comparing
\eqref{eq:Rejective}  and \eqref{eq:Poisson}  reveals that rejective $R_{N}$
sampling of fixed size $n$ can be viewed as Poisson sampling given that the
sample size is equal to $n$. It is for this reason that rejective sampling is
usually referred to as \textit{conditional Poisson sampling}. For simplicity's
sake, the superscrit $R$ is omitted in the sequel. One must pay attention not
to get the $\pi_{i}$'s and the $p_{i}$'s mixed up (except in the SWOR case, where these quantities are all equal to $n/N$): the latter are the first
order inclusion probabilities of $P_{N}$, whereas the former are those of its
conditional version $R_{N}$. However they can be related by means of the
results stated in \cite{Hajek64} (see Theorem 5.1 therein, as well as Lemma \ref{lem:bias} in section \ref{sec:main} and \cite{BLRG12}): $\forall
i\in\{1,\;\ldots,\;N\}$,
\begin{align}
\pi_{i}(1-p_{i})  &  =p_{i}(1-\pi_{i})\times\left(  1-\left(  \tilde{\pi}%
-\pi_{i}\right)  /d_{N}^{\ast}+o(1/d_{N}^{\ast})\right)  ,\label{eq:rel1}\\
p_{i}(1-\pi_{i})  &  =\pi_{i}(1-p_{i})\times\left(  1-\left(  \tilde{p}%
-p_{i}\right)  /d_{N}+o(1/d_{N})\right)  , \label{eq:rel2}%
\end{align}
where $d_{N}^{\ast}=\sum_{i=1}^{N}\pi_{i}(1-\pi_{i})$, $d_{N}=\sum_{i=1}%
^{N}p_{i}(1-p_{i})$, $\tilde{\pi}=(1/d_{N}^{\ast})\sum_{i=1}^{N}\pi_{i}%
^{2}(1-\pi_{i})$ and $\tilde{p}=(1/d_{N})\sum_{i=1}^{N}(p_{i})^{2}(1-p_{i})$. \medskip

Since the major advantage of conditional Poisson sampling lies in its reduced
variance property (compared to Poisson sampling in particular, see the discussion in section \ref{sec:main}), focus is next
on exponential inequalities involving a variance term, of Bennett/Bernstein
type namely.

\section{Preliminary Results}

\label{sec:Poisson}

As a first go, we establish tail bounds for the Horvitz-Thompson estimator in
the case where the variables are sampled according to a Poisson scheme. We
next show how to exploit the \textit{negative association} property satisfied
by rejective sampling in order to extend the latter to conditional Poisson
sampling. Of course, this approach do not account for the reduced variance
property of Horvitz-Thompson estimates based on rejective sampling, it is the
purpose of the next section to improve these first exponential bounds.

\subsection{Tails bounds for Poisson sampling}

As previously observed, bounding the tail probability \eqref{eq:tailprob}  is
easy in the Poisson situation insofar as the variables summed up in
\eqref{eq:HT}  are independent though possibly non identically distributed
(since the inclusion probabilities are not assumed to be all equal). The
following theorem thus directly follows from well-known results related to
tail bounds for sums of independent random variables.

\begin{theorem}
\label{thm:poisson}\textsc{(Poisson sampling)} Assume that the survey scheme
$\boldsymbol{\epsilon} _{N}$ defines a Poisson sampling plan with first order
inclusion probabilities $p_{i}>0$, with $1\leq i \leq N$. Then, we
almost-surely have: $\forall t>0$, $\forall N\geq1$,
\begin{align}
\mathbb{P}\left\{  \widehat{S}_{\boldsymbol{p} _{N}}^{\boldsymbol{\epsilon}
_{N}}-S_{N} >t \right\}   &  \leq\exp\left(  -\frac{\sum_{i=1}^{N}%
\frac{1-p_{i}}{p_{i}}x_{i}^{2}}{\left(  \max_{1\leq i \leq N}\frac{x_{i}%
}{p_{i}} \right)  ^{2}} H\left(  \frac{\max_{1\leq i \leq N}\frac{\vert
x_{i}\vert}{p_{i}}t}{\sum_{i=1}^{N}\frac{1-p_{i}}{p_{i}}x_{i}^{2}} \right)
\right) \label{eq:Bennett}\\
&  \leq\exp\left(  \frac{-t^{2}}{\frac{2}{3}\max_{1\leq i\leq N}\frac{\vert
x_{i}\vert}{p_{i}}+ 2\sum_{i=1}^{N}\frac{1-p_{i}}{p_{i}}x_{i}^{2}} \right)  ,
\label{eq:Bern}%
\end{align}
where $H(x)=(1+x)\log(1+x)-x$ for $x\geq0$.
\end{theorem}

Bounds \eqref{eq:Bennett}  and \eqref{eq:Bern}  straightforwardly result from
Bennett inequality \cite{Bennett} and Bernstein exponential inequality
\cite{Bernstein} respectively, when applied to the independent random
variables $(\epsilon_{i}/p_{i})x_{i}$, $1\leq i \leq N$. By applying these
results to the variables $-(\epsilon_{i}/p_{i})x_{i}$'s, the same bounds
naturally hold for the deviation probability $\mathbb{P}\{\widehat
{S}_{\boldsymbol{p} _{N}}^{\boldsymbol{\epsilon} _{N}}-S_{N} <-t \}$ (and,
incidentally, for $\mathbb{P}\{\vert\widehat{S}_{\boldsymbol{p} _{N}%
}^{\boldsymbol{\epsilon} _{N}}-S_{N} \vert>t \}$ up to a factor $2$).
Details, as well as extensions to other deviation inequalities (see
\textit{e.g.} \cite{FukNagaev}), are left to the reader.


\subsection{Exponential inequalities for sums of negatively associated random variables}

For clarity, we first recall the definition of \textit{negatively associated
random variables}, see \cite{JDP83}.

\begin{definition}
\label{def:negassoc} Let $Z_{1},\; \ldots,\; Z_{n}$ be random variables
defined on the same probability space, valued in a measurable space
$(E,\mathcal{E})$. They are said to be negatively associated iff for any pair
of disjoint subsets $A_{1}$ and $A_{2}$ of the index set $\{1,\; \ldots,\; n
\}$
\begin{equation}
\label{eq:neg}Cov \left(  f((Z_{i})_{i\in A_{1}}),\; g((Z_{j})_{j\in A_{2}})
\right)  \leq0,
\end{equation}
for any real valued measurable functions $f:E^{\#A_{1}}\rightarrow\mathbb{R}$
and $g:E^{\#A_{2}}\rightarrow\mathbb{R}$ that are both increasing in each variable.
\end{definition}

The following result provides tail bounds for sums of negatively associated
random variables, which extends the usual Bennett/Bernstein inequalities in the
i.i.d. setting, see \cite{Bennett} and \cite{Bernstein}.

\begin{theorem}
\label{thm:BernNeg} Let $Z_{1},\;\ldots,\;Z_{N}$ be square integrable
negatively associated real valued random variables such that $|Z_{i}|\leq c$
a.s. and $\mathbb{E}[Z_{i}]=0$ for $1\leq i\leq N$. Let $a_{1},\; \ldots,\; a_N$ be
non negative constants and set $\sigma^{2}=\frac{1}{N}\sum_{i=1}^{N}a_{i}%
^{2}Var(Z_{i})$. Then, for all $t>0$, we have: $\forall N\geq1$,
\begin{align}
\mathbb{P}\left\{  \sum_{i=1}^{N}a_{i}Z_{i}\geq t\right\}   &  \leq\exp\left(
-\frac{N\sigma^{2}}{c^{2}}H\left(  \frac{ct}{N\sigma^{2}}\right)  \right) \\
&  \leq\exp\left(  -\frac{t^{2}}{2N\sigma^{2}+\frac{2ct}{3}}\right)  .
\end{align}
\end{theorem}

Before detailing the proof, observe that the same bounds hold true for the
tail probability $\mathbb{P}\left\{ \sum_{i=1}^{N}a_{i}Z_{i}\leq-t\right\}  $
(and for $\mathbb{P}\left\{  |\sum_{i=1}^{N}a_{i}Z_{i}|\geq t\right\}  $ as
well, up to a multiplicative factor $2$). Refer also to Theorem 4 in
\cite{Janson} for a similar result in a more restrictive setting
(\textit{i.e.} for tail bounds related to sums of \textit{negatively related}
r.v.'s) and to \cite{Shao00} as well. \begin{proof}%

The proof starts off with the usual Chernoff method: for all $\lambda>0$,
\begin{equation}%
\label{eq:Chernoff}
\mathbb{P}\left\{\sum_{i=1}
^N a_i Z_i  \geq t\right\}\leq \exp\left( -t\lambda +\log \mathbb{E}%
\left[e^{t\sum_{i=1}^N a_i Z_i} \right] \right).
\end{equation}

Next, observe that, for all $t>0$, we have
\begin{eqnarray}\label{eq:neg2}
\mathbb{E}\left[\exp\left(t\sum_{i=1}^n a_i Z_i\right)\right]&=&\mathbb{E}
\left[\exp(t a_n Z_n )\exp\left(t\sum_{i=1}^{n-1}
a_i Z_i\right)\right]\nonumber\\
&\leq &\mathbb{E}
\left[ \exp(ta_n Z_n) \right]\mathbb{E}\left[\exp\left(t\sum_{i=1}^{n-1}
a_i Z_i\right)  \right]\nonumber\\
&\leq & \prod_{i=1}^n\mathbb{E}
\left[ \exp(ta_iZ_i) \right],\label{eq:neg2}
\end{eqnarray}

using the property \eqref{eq:neg}
combined with a descending recurrence on $i$. The proof is finished by plugging \eqref{eq:neg2}
into \eqref{eq:Chernoff}
and optimizing finally the resulting bound w.r.t. $\lambda>0$, just like in the proof of the classic Bennett/Bernstein inequalities, see \cite{Bennett}
and \cite{Bernstein}. $\square$
\end{proof} \medskip

The first assertion of the theorem stated below reveals that any rejective
scheme $\boldsymbol{\epsilon} ^{*}_{N}$ forms a collection of negatively
related r.v.'s, the second one appearing then as a direct consequence of Theorem \ref{thm:BernNeg}.
We underline that many sampling schemes (\textit{e.g.} Rao-Sampford sampling, Pareto sampling, Srinivasan sampling) of fixed size are actually described by random vectors $\boldsymbol{\epsilon}_N$ with negatively associated components, see \cite{BJ12} or \cite{KCR11}, so that exponential bounds similar to that stated below can be proved for such sampling plans.

\begin{theorem}
\label{thm:neg} Let $N\geq1$ and $\boldsymbol{\epsilon} ^{*}_{N}=(\epsilon
^{*}_{1},\; \ldots,\; \epsilon^{*}_{N})$ be the vector of indicator variables
related to a rejective plan on $\mathcal{I}_{N}$ with first order inclusion probabilities $(\pi_1,\; \ldots,\; \pi_N)\in ]0,1]^N$. Then, the following
assertions hold true.

\begin{itemize}
\item[(i)] The binary random variables $\epsilon^{*}_{1},\; \ldots,\;
\epsilon^{*}_{N}$ are negatively related.

\item[(ii)] For any $t\geq0$ and $N\geq1$, we have:
\begin{align*}
\mathbb{P}\left\{  \widehat{S}_{\boldsymbol{\pi} }^{\boldsymbol{\epsilon}
^{*}_{N}}-S_{N} \geq t \right\}   &  \leq2 \exp\left(  -\frac{\sum_{i=1}%
^{N}\frac{1-\pi_{i}}{\pi_{i}}x_{i}^{2}}{\left(  \max_{1\leq i \leq
N}\frac{x_{i}}{\pi_{i}} \right)  ^{2}} H\left(  \frac{\max_{1\leq i \leq
N}\frac{\vert x_{i}\vert}{\pi_{i}}t/2}{\sum_{i=1}^{N}\frac{1-\pi_{i}}{\pi_{i}%
}x_{i}^{2}} \right)  \right) \\
&  \leq2 \exp\left(  \frac{-t^{2}/4}{\frac{2}{3}\max_{1\leq i\leq N}\frac{\vert
x_{i}\vert}{\pi_{i}}t+ 2\sum_{i=1}^{N}\frac{1-\pi_{i}}{\pi_{i}}x_{i}^{2}}
\right)  .
\end{align*}
\end{itemize}
\end{theorem}

\begin{proof}%

Considering the usual representation of the distribution of $(\epsilon_1,\; \ldots,\; \epsilon_N)$ as the conditional distribution of a sample of independent Bernoulli variables $(\epsilon^*_1,\; \ldots,\; \epsilon^*_N)$ conditioned upon the event $\sum_{i=1}%
^N\epsilon^*_i=n$ (see subsection \ref{subsec:Poisson}%
), Assertion $(i)$ is a straightforward consequence from Theorem 2.8 in \cite{JDP83}
(see also \cite{Barbour}%
).
\medskip
Assertion $(i)$ shows in particular that Theorem \ref{thm:BernNeg}
can be applied to the random variables $\{ (\epsilon_i^*/\pi_i-1)x_i^+:\; 1\leq i \leq N \}$ and to the random variables $\{ (\epsilon_i^*/\pi_i-1)x_i^-:\; 1\leq i \leq N \}$ as well. Using the union bound, we obtain that
\begin{multline*}
\mathbb{P}\left\{ \widehat{S}_{\boldsymbol{\pi}}^{\boldsymbol{\epsilon}%
^*_N}-S_N \geq t \right\}\leq \mathbb{P}\left\{ \sum_{i=1}%
^N\left( \frac{\epsilon^*_i}{\pi_i}%
-1 \right)x^+_i\geq t/2 \right\} \\+ \mathbb{P}\left\{ \sum_{i=1}%
^N\left( \frac{\epsilon^*_i}{\pi_i}%
-1 \right)x^-_i\leq -t/2 \right\},
\end{multline*}
and a direct application of Theorem \ref{thm:BernNeg} to each of the terms involved in this bound straightforwardly proves Assertion $(ii)$. $\square$
\end{proof}
 \medskip

The negative association property permits to handle the dependence of the
terms involved in the summation. However, it may lead to rather loose
probability bounds. Indeed, except the factor $2$, the bounds of Assertion
$(ii)$ exactly correspond to those stated in Theorem \ref{thm:poisson}, as if
the $\epsilon_{i}^{*}$'s were independent, whereas the asymptotic variance $\sigma^2_N$ of
$\widehat{S}_{\boldsymbol{\pi} }^{\boldsymbol{\epsilon} _{N}^{*}}$ can be much smaller than $\sum_{i=1}^{N}(1-\pi_{i})x_{i}^{2}/\pi_{i}$.
It is the goal of the subsequent analysis to improve these preliminary results and establish exponential bounds involving the asymptotic variance $\sigma^2_N$.
\begin{remark}{\sc (SWOR)} We point out that in the specific case of sampling without replacement, \textit{i.e.} when $\pi_i=n/N$ for all $i\in \{1,\; \ldots,\; N \}$, the inequality stated in Assertion $(ii)$ is quite comparable (except the factor $2$) to that which can be derived from the Chernoff bound given in \cite{Hoeffding63}, see Proposition 2 in \cite{BardenetMaillard}.

\end{remark}


\section{Main Results - Exponential Inequalities for Rejective Sampling}

\label{sec:main} The main results of the paper are stated and discussed in the
present section. More accurate deviation probabilities related to the total estimate
\eqref{eq:HT}  based on a rejective sampling scheme $\boldsymbol{\epsilon}
_{N}^{\ast}$ of (fixed) sample size $n\leq N$ with first order inclusion
probabilities $\boldsymbol{\pi} _{N}=(\pi_{1},\;\ldots,\;\pi_{N})$ and
canonical representation $\mathbf{p}_{N}=(p_{1},\;\ldots,\;p_{N})$ are now
investigated. Consider $\boldsymbol{\epsilon} _{N}$ a Poisson scheme with
$\mathbf{p}_{N}$ as vector of first order inclusion probabilities. As
previously recalled, the distribution of $\boldsymbol{\epsilon} _{N}^{\ast}$
is equal to the conditional distribution of $\boldsymbol{\epsilon} _{N}$ given
$\sum_{i=1}^{N}\varepsilon_{i}=n$:
\begin{equation}
(\varepsilon_{1}^{\ast},\varepsilon_{2}^{\ast},....,\varepsilon_{N}^{\ast
})\overset{d}{=}(\varepsilon_{1},....,\varepsilon_{N})\ |\sum_{i=1}%
^{N}\varepsilon_{i}=n.\label{eq:distribution}%
\end{equation}
Hence, we almost-surely have: $\forall t>0$, $\forall N\geq1$,
\begin{equation}
\mathbb{P}\left\{  \widehat{S}_{\boldsymbol{\pi} _{N}}%
^{\boldsymbol{\epsilon} _{N}^{\ast}}-S_{N}>t\right\}  =\mathbb{P}\left\{  \sum_{i=1}^{N}\frac{\epsilon_{i}}{\pi_{i}}x_{i}-S_{N}>t\mid
\sum_{i=1}^{N}\epsilon_{i}=n\right\}  .\label{eq:cond_tail}%
\end{equation}
As a first go, we shall prove tail bounds for the quantity
\begin{equation}
\widehat{S}_{\boldsymbol{p} _{N}}^{\boldsymbol{\epsilon} _{N}^{\ast}}%
\overset{def}{=}\sum_{i=1}^{N}\frac{\epsilon_{i}^{\ast}}{p_{i}}x_{i}%
.\label{eq:HT2}%
\end{equation}
Observe that this corresponds to the HT estimate of the total $\sum_{i=1}%
^{N}\frac{p_{i}}{\pi_{i}}x_{i}$. Refinements of relationships \eqref{eq:rel1}  and
\eqref{eq:rel2}  between the $p_{i}$'s and the $\pi_{i}$'s shall next allow us
to obtain an upper bound for \eqref{eq:cond_tail}. Notice incidentally that,
though slightly biased (see Assertion $(i)$ of Theorem \ref{thm:final}), the statistic
\eqref{eq:HT2}  is commonly used as an estimator of $S_{N}$, insofar as the
parameters $p_{i}$'s are readily available from the canonical representation
of $\boldsymbol{\epsilon} _{N}^{\ast}$, whereas the computation of the
$\pi_{i}$'s is much more complicated. One may refer to \cite{CDL94} for
practical algorithms dedicated to this task. Hence, Theorem \ref{thm:rejective} is of practical interest to build non asymptotic confidence intervals for the total $S_N$.
\medskip

\noindent {\bf Asymptotic variance.} Recall that $d_{N}=\sum_{i=1}^{N}p_{i}(1-p_{i})$ is the variance
$Var(\sum_{i=1}^{N}\epsilon_{i})$ of the size of the Poisson plan
$\boldsymbol{\epsilon} _{N}$ and set
\[
\theta_{N}=\frac{\sum_{i=1}^{N}x_{i}(1-p_{i})}{d_{N}}.
\]
As explained in \cite{BCC2013}, the quantity $\theta_{N}$ is the coefficient of the linear regression relating $\sum_{i=1}^{N}\frac{\epsilon_{i}}{p_{i}}x_{i}-S_{N}$ to the sample
size $\sum_{i=1}^{N}\epsilon_{i}$. We may thus write
\[
\sum_{i=1}^{N}\frac{\epsilon_{i}}{p_{i}}x_{i}-S_{N}=\theta_{N}\times \sum_{i=1}%
^{N}\epsilon_{i}+r_{N},
\]
where the residual $r_{N}$\ is orthogonal to $\sum
_{i=1}^{N}\epsilon_{i}$. Hence, we have the following decomposition
\begin{equation}\label{eq:Poisson_var}
Var\left(  \sum_{i=1}^{N}\frac{\epsilon_{i}}{p_{i}}x_{i}\right)  =\sigma^2_{N}+\theta_{N}^{2}d_{N},
\end{equation}
where 
\begin{equation}\label{eq:Asympt_Var}
\sigma^2_{N}=Var\left(  \sum_{i=1}^{N}(\epsilon_{i}-p_{i})\left(
\frac{x_{i}}{p_{i}}-\theta_{N}\right)  \right)  
\end{equation}
 is the asymptotic variance
of the statistic $\widehat{S}_{\boldsymbol{p} _{N}}^{\boldsymbol{\epsilon} _{N}^{\ast}}$, see \cite{Hajek64}. In other words,
the variance reduction resulting from the use of a rejective sampling plan instead of a Poisson plan is
equal to $\theta_{N}^{2}d_{N}$, and can be very large in practice. A sharp Bernstein type probability inequality for $\widehat{S}_{\boldsymbol{p} _{N}}^{\boldsymbol{\epsilon} _{N}^{\ast}}$ should thus involve $\sigma^2_N$ rather than the Poisson variance $Var(  \sum_{i=1}^{N}(\epsilon_{i}/p_{i})x_{i})$.
Using the fact that $\sum_{i=1}^{N}(\epsilon_{i}-p_{i})=0$ on the event
$\{\sum_{i=1}^{N}\epsilon_{i}=n\}$, we may now write:
\begin{align}
\mathbb{P}\left\{  \widehat{S}_{\boldsymbol{p} _{N}}^{\boldsymbol{\epsilon}
_{N}^{\ast}}-S_{N}>t\right\}   &  =\mathbb{P}\left\{  \sum_{i=1}%
^{N}\frac{\epsilon_{i}}{p_{i}}x_{i}-S_{N}>t\mid\sum_{i=1}^{N}\epsilon
_{i}=n\right\}  \nonumber\label{eq:ratio}\\
&  =\frac{\mathbb{P}\left\{  \sum_{i=1}^{N}(\epsilon_{i}-p_{i})\frac{x_{i}%
}{p_{i}}>t,\;\sum_{i=1}^{N}\epsilon_{i}=n\right\}  }{\mathbb{P}\left\{
\sum_{i=1}^{N}\epsilon_{i}=n\right\}  }\nonumber\\
&  =\frac{\mathbb{P}\left\{  \sum_{i=1}^{N}(\epsilon_{i}-p_{i})\left(
\frac{x_{i}}{p_{i}}-\theta_{N}\right)  >t,\sum_{i=1}^{N}\epsilon
_{i}=n\right\}  }{\mathbb{P}\left\{  \sum_{i=1}^{N}\epsilon_{i}=n\right\}  }.
\end{align}
Based on the
observation that the random variables $\sum_{i=1}^{N}(\epsilon_{i}%
-p_{i})(x_{i}/p_{i}-\theta_{N})$ and $\sum_{i=1}^{N}(\epsilon_{i}-p_{i})$ are
uncorrelated, Eq. \eqref{eq:ratio} thus permits to establish directly the CLT $\sigma_N^{-1}(\widehat{S}_{\boldsymbol{p} _{N}}^{\boldsymbol{\epsilon} _{N}^{\ast}}-S_{N})\Rightarrow \mathcal{N}(0,1)$,
provided that $d_{N}\rightarrow+\infty$, as $N\rightarrow+\infty$, symplifying
asymptotically the ratio, see \cite{Hajek64}. Hence, the asymptotic variance of $\widehat{S}_{\boldsymbol{p} _{N}}^{\boldsymbol{\epsilon} _{N}^{\ast}}-S_{N}$ is the variance $\sigma^2_N$ of the quantity $\sum_{i=1}^{N}(\epsilon_{i}%
-p_{i})(x_{i}/p_{i}-\theta_{N})$, which is less than that of the Poisson HT estimate $\eqref{eq:Poisson_var}$, since it eliminates the variability due to the sample size. We also point out that Lemma \ref{lem:bias} proved in the Appendix section straightforwardly shows that the "variance term" $\sum_{i=1}^Nx_i^2(1-\pi_i)/\pi_i$ involved in the bound stated in Theorem \ref{thm:BernNeg} is always larger than $(1+6/d_N)^{-1}\sum_{i=1}^Nx_i^2(1-p_i)/p_i$.
\medskip

The desired result here is non
asymptotic and accurate exponential bounds are required for both the numerator
and the denominator of \eqref{eq:ratio}. It is proved in \cite{Hajek64} (see
Lemma 3.1 therein) that, as $N\rightarrow+\infty$:
\begin{equation}
\mathbb{P}\left\{  \sum_{i=1}^{N}\epsilon_{i}=n\right\}  =(2\,\pi
\,d_{N})^{-1/2}\,(1+o(1)).\label{local}%
\end{equation}
As shall be seen in the proof of the theorem stated below, the approximation
\eqref{local}  can be refined by using a local Berry-Essen bound or the
results in \cite{Deheuvels} and we thus essentially need to establish an
exponential bound for the numerator with a constant of order $d_{N}^{-1/2}$,
sharp enough so as to simplify the resulting ratio bound and cancel off the
denominator. We shall prove that this can be achieved by using a similar
argument as that considered in \cite{BERCLEM2010} for establishing an accurate
exponential bound for i.i.d. $1$-lattice random vectors, based on a device
introduced in \cite{Talagrand95} for refining Hoeffding's inequality.

\begin{theorem}
\label{thm:rejective}Let $N\geq1$. Suppose that $\boldsymbol{\epsilon}
_{N}^{\ast}$ is a rejective scheme of size $n\leq N$ with canonical parameter
$\boldsymbol{p} _{N}=(p_{1},\;\ldots,\;p_{N})\in]0,\;1[^{N}$. Then, there
exist universal constants $C$ and $D$ such that we have for all $t>0$ and
for all $N\geq1$,
\begin{align*}
\mathbb{P}\left\{  \widehat{S}_{\boldsymbol{p} _{N}}^{\boldsymbol{\epsilon}
_{N}^{\ast}}-S_{N}>t\right\}   &  \leq C\exp\left(  -\frac{\sigma^2_{N}%
}{\left(  \max_{1\leq j\leq N}\frac{|x_{j}|}{p_{j}}\right)  ^{2}}H\left(
\frac{t\max_{1\leq j\leq N}\frac{|x_{j}|}{p_{j}}}%
{\sigma^2_{N}}\right)  \right)  \\
&  \leq C\exp\left(  -\frac{t^{2}}{2\left(  \sigma^2_{N}+\frac{1}%
{3}t\max_{1\leq j\leq N}\frac{|x_{j}|}{p_{j}}\right)
}\right)  ,
\end{align*}
as soon as $\min\{d_{N},\;d_{N}^{\ast}\}\geq1$ and $d_N\geq D$.
\end{theorem}

An overestimated value of the constant $C$ can be deduced by a careful
examination of the proof given below. Before we detail it, we point out that
the exponential bound in Theorem \ref{thm:rejective} involves the asymptotic variance of \eqref{eq:HT2}, in contrast to bounds
obtained by exploiting the \textit{negative association} property of the
$\epsilon_{i}^{\ast}$'s.

\begin{remark}{\sc (SWOR (bis))} We underline that, in the particular case of sampling without replacement (\textit{i.e.} when $p_i=\pi_i=n/N$ for $1\leq i\leq N$),  the Bernstein type exponential inequality stated above provides a control of the tail similar to that obtained in \cite{BardenetMaillard}, see Theorem 2 therein, with $k=n$. In this specific situation, we have $d_N=n(1-n/N)$ and $\theta_N=S_N/n$, so that the formula \eqref{eq:Asympt_Var} then becomes $$
\sigma_N^2=\left(1-\frac{n}{N}\right)\frac{N^2}{n}\left\{ \frac{1}{N}\sum_{i=1}^Nx_i^2 -\left(\frac{1}{N}\sum_{i=1}^N x_i\right)^2 \right\}.
$$
The control induced by Theorem \ref{thm:rejective} is actually slightly better than that given by Theorem 2 in \cite{BardenetMaillard}, insofar as the factor $(1-n/N)$ is involved in the variance term, rather than $(1-(n-1)/N)$, that is crucial when considering situations where $n$ gets close to $N$ (see the discussion preceded by Proposition 2 in \cite{BardenetMaillard}).
\end{remark}

\begin{proof}
We first introduce additional notations.
Set $Z_{i}%
=(\epsilon _{i}-p_{i})(x_{i}/p_{i}-\theta _{N})$ and $m_{i}=\epsilon _{i}%
-p_{i}%
$ for $1\leq i\leq N$ and, for convenience, consider the standardized variables given by
\begin{equation*}
\mathcal{Z}_{N}=n^{1/2}\frac{1}{N}\sum_{1\leq i\leq N}Z_{i} \text{ and }
\mathcal{M}_{N}=d_N^{-1/2}\sum_{1\leq i\leq N}m_{i}.
\end{equation*}%
As previously announced, the proof is based on Eq. \eqref{eq:ratio}. The lemma below first provides a sharp lower bound for the denominator, $ \mathbb{P}%
^*\left\{ \mathcal{M}_{N}%
=0\right\}$ with the notations above. As shown in the proof given in the Appendix section, it can be obtained by applying the local Berry-Esseen bound established in \cite{Deheuvels}
for sums of independent (and possibly non identically) Bernoulli random variables.
\begin{lemma}
\label{lem:denominator}
Suppose that Theorem \ref{thm:rejective}%
's assumptions are fulfilled. Then, there exist universal constants $C_1$ and $D$ such that: $\forall N\geq 1$,
\begin{equation}
\mathbb{P}\{ \mathcal{M}_N=0 \}\geq C_1\frac{1}{\sqrt{d_N}},
\end{equation}
provided that $d_N\geq D$.
\end{lemma}

The second lemma gives an accurate upper bound for the numerator. Its proof can be found in the Appendix section.
\begin{lemma}%
\label{lem:numerator}
Suppose that Theorem \ref{thm:rejective}%
's assumptions are fulfilled. Then, we have for all $x\geq 0$, and for all $N\geq 1$ such that $\min\{d_N,\; d_N^*\}\geq 1$:
\begin{multline*}
\mathbb{P}\left\{\mathcal{Z}_{N}\geq  x,\mathcal{M}_{N}%
=0\right\}\leq C \frac{1}{\sqrt{d_N}}%
\times \\\exp\left( -\frac{Var\left( \sum_{i=1}^NZ_i \right)}%
{\left(\max_{1\leq j\leq N}\frac{\vert x_j\vert}{p_j} \right)^2}%
h\left( \frac{N}{\sqrt{n}}\frac{x\max_{1\leq j\leq N}\frac{\vert x_j\vert}%
{p_j}}{Var\left( \sum_{i=1}^NZ_i \right)}
\right) \right)\\
\leq  C_2 \frac{1}{\sqrt{d_N}}\exp\left( -\frac{N^2x^2/n}%
{2\left(Var\left( \sum_{i=1}^NZ_i \right)+\frac{1}{3}\frac{N}{\sqrt{n}%
}x\max_{1\leq j\leq N}\frac{\vert x_j\vert}{p_j} \right)}
\right),
\end{multline*}%
where $C_2<+\infty$ is a universal constant.
\end{lemma}

The bound stated in Theorem \ref{thm:rejective}
now directly results from Eq. \eqref{eq:ratio}
combined with Lemmas \ref{lem:denominator} and \ref{lem:numerator}, with $x=t\frac{\sqrt{n}}{N}$. $\square$
\end{proof} 

\bigskip

Even if the computation of the biased statistic \eqref{eq:HT2} is much more tractable from a practical perspective, we now come back to the study of the HT total estimate \eqref{eq:HT}. The first part of the result stated below provides an estimation of the bias that replacement of \eqref{eq:HT} by \eqref{eq:HT2} induces, whereas its second part finally gives a tail bound for $\eqref{eq:HT}$.

\begin{theorem}\label{thm:final} Suppose that the assumptions of Theorem \ref{thm:rejective} are fulfilled and set $M_N=(6/d_N)\sum_{i=1}^N\vert x_i\vert /\pi_i$. The following assertions hold true.
\begin{itemize}
\item[(i)] For all $N\geq 1$, we almost-surely have:
\begin{equation*}
\left\vert \widehat{S}^{\boldsymbol{\epsilon}_N^*}_{\boldsymbol{\pi}_N } - \widehat{S}^{\boldsymbol{\epsilon}_N^*}_{\mathbf{p}_N } \right\vert \leq M_N .
\end{equation*}
\item[(ii)] There exist universal
constants $C$ and $D$ such that, for all $t>M_{N}$ and for all $N\geq1$, we have:
\begin{multline*}
\mathbb{P}\left\{  \widehat{S}^{\boldsymbol{\epsilon}^*_N}_{\boldsymbol{\pi}_N}-S_{N}>t \right\}   
\leq \\ C\exp\left(  -\frac{\sigma_{N}^2}{\left(  \max_{1\leq j\leq
N}\frac{|x_{j}|}{p_{j}}\right)  ^{2}}H\left(  \frac{N}{\sqrt{n}}%
\frac{(t-M_{N})\max_{1\leq j\leq N}\frac{|x_{j}|}{p_{j}}}{\sigma^2_{N}%
}\right)  \right)  \\
  \leq  C\exp\left(  -\frac{N^{2}(t-M_{N})^{2}/n}{2\left(  \sigma^2_{N}+\frac{1}{3}\frac{N}{\sqrt{n}}(t-M_{N})\max_{1\leq j\leq N}\frac{|x_{j}%
|}{p_{j}}\right)  }\right)  ,
\end{multline*}
as soon as $\min\{d_{N},\;d_{N}^{\ast}\}\geq1$ and $d_N\geq D$.
\end{itemize}
\end{theorem}

The proof is given in the Appendix section. We point out that, for nearly uniform weights, \textit{i.e.} when $c_1n/N\leq\pi_i\leq c_2n/N$ for all $i\in\{1,\; \ldots,\; N  \}$ with $0<c_1\leq c_2<+\infty$, if there exists $K<+\infty$ such that $\max_{1\leq i\leq N}\vert x_i\vert \leq K$ for all $N\geq 1$, then the bias term $M_N$ is of order $o(N)$, provided that $\sqrt{N}/n\rightarrow 0$ as $N\rightarrow +\infty$.

\section{Extensions to more general sampling schemes}\label{sec:extension}

We finally explain how the results established in the previous section for rejective sampling may permit to control tail probabilities for more general sampling plans. A similar argument is used in \cite{Ber98} to derive CLT's for HT estimators based on complex sampling schemes that can be approximated by more simple sampling plans, see also \cite{BCC2013}. 
Let $\widetilde{R}_{N}$ and $R_{N}$ be two sampling plans on the population $\mathcal{I}_N$ and consider the
\emph{total variation metric}
\[
\Vert\widetilde{R}_{N}-R_{N}\Vert_{1}\overset{def}{=}\sum_{s\in \mathcal{P}(\mathcal{I}_{N})}\left|
\widetilde{R}_{N}(s)-R_{N}(s)\right|  ,
\]
as well as the \emph{Kullback-Leibler divergence}
\[
D_{KL}(R_{N}\vert\vert \widetilde{R}_{N})\overset{def}{=}\sum_{s\in \mathcal{P}(\mathcal{I}_{N})}R_{N}%
(s)\,\log\left(  \frac{R_{N}(s)}{\widetilde{R}_{N}(s)}\right)  .
\]
Equipped with these notations, we can state the following result.
\begin{lemma} \label{lem:ext} Let $\boldsymbol{\epsilon}_N$ and $\widetilde{\boldsymbol{\epsilon}}_N$ be two schemes defined on the same probability space and drawn from plans $R_N$ and $\widetilde{R}_N$ respectively and let $\mathbf{p}_N\in ]0,1]^N$. Then, we have: $\forall N\geq 1$, $\forall t\in \mathbb{R}$,
\begin{eqnarray*}
 \left|  \mathbb{P}\left\{  \widehat{S}_{\mathbf{p} _{N}%
}^{\boldsymbol{\epsilon} _{N}}-S_{N}>t\right\}  -\mathbb{P}\left\{  \widehat{S}_{\mathbf{p} _{N}}%
^{\widetilde{\boldsymbol{\epsilon}}_{N}}-S_{N}>t\right\}  \right| 
& \leq &\Vert\widetilde{R}_{N}-R_{N}\Vert_{1} \\ &\leq&\sqrt{2 D_{KL}(R_{N}\vert\vert\widetilde
{R}_{N})}.
\end{eqnarray*}
\end{lemma}
\begin{proof} The first bound immediately results from the following elementary observation:
\begin{multline*}
 \mathbb{P}\left\{  \widehat{S}_{\mathbf{p} _{N}%
}^{\boldsymbol{\epsilon} _{N}}-S_{N}>t\right\}  -\mathbb{P}\left\{  \widehat{S}_{\mathbf{p} _{N}}%
^{\widetilde{\boldsymbol{\epsilon}}_{N}}-S_{N}>t\right\} =\\
\sum_{s\in \mathcal{P}(\mathcal{I}_N)}\mathbb{I}\{\sum_{i\in s}x_i/p_i-S_N  >t\}\times \left(R_N(s) -\widetilde{R}_N(s) \right),
\end{multline*}
while the second bound is the classical Pinsker inequality.
$\square$
\end{proof}
\medskip

In practice, $R_{N}$ is typically the rejective sampling plan
investigated in the previous subsection (or eventually the Poisson sampling scheme) and $\widetilde{R}_N$ a sampling plan from which the Kullback-Leibler divergence to $R_N$ asymptotically vanishes, \textit{e.g.} the rate at which $D_{KL}(R_{N}\vert\vert\widetilde{R}_{N})$ decays to zero has been investigated in \cite{Ber98} when $\widetilde{R}_N$ corresponds to Rao-Sampford, successive sampling or Pareto sampling
under appropriate regular conditions (see also \cite{BTL06}). Lemma \ref{lem:ext} combined with Theorem \ref{thm:rejective} or Theorem \ref{thm:final} permits then to obtain upper bounds for the tail probabilities $\mathbb{P}\{  \widehat{S}_{\mathbf{p} _{N}}%
^{\widetilde{\boldsymbol{\epsilon}}_{N}}-S_{N}>t\} $.

\section{Conclusion}\label{sec:concl}
In this article, we proved Bernstein-type tail bounds to quantify the deviation between a total and its Horvitz-Thompson estimator when based on conditional Poisson sampling, extending (and even slightly improving) results proved in the case of basic sampling without replacement. The original proof technique used to establish these inequalities relies on expressing the deviation probablities related to a conditional Poisson scheme as conditional probabilities related to a Poisson plan. This permits to recover tight exponential bounds, involving the asymptotic variance of the Horvitz-Thompson estimator. Beyond the fact that rejective sampling is of prime importance in the practice of survey sampling, extension of these tail bounds to sampling schemes that can be accurately approximated by rejective sampling in the total variation sense is also discussed.

\section*{Appendix - Technical Proofs}

\subsection*{Proof of Lemma \ref{lem:denominator}}

For clarity, we first recall the following result.

\begin{theorem}
\label{thm:denominator}(\cite{Deheuvels}, Theorem 1.3) Let $(Y_{j,n})_{1\leq
j\leq n}$ be a triangular array of independent Bernoulli random variables with
means $q_{1,n},\; \ldots,\; q_{n,n}$ in $(0,1)$ respectively. Denote by
$\sigma^{2}_{n}=\sum_{i=1}^{n}q_{i,n}(1-q_{i,n})$ the variance of the sum
$\Sigma_{n}=\sum_{i=1}^{n}Y_{i,n}$ and by $\nu_{n}=\sum_{i=1}^{n}q_{i,n}$ its
mean. Considering the cumulative distribution function (cdf) $F_{n}%
(x)=\mathbb{P}\{ \sigma_{n}^{-1}(\Sigma_{n}-\nu_{n} )\leq x \}$, we have:
$\forall n\geq1$,
\[
\sup_{k\in\mathbb{Z}}\left\vert F_{n}(x_{n,k})-\Phi(x_{n,k})-\frac{1-x_{n,k}%
^{2}}{6\sigma_{n}}\phi(x_{n,k})\left\{  1-\frac{2\sum_{i=1}^{n}q^{2}%
_{i,n}(1-q_{i,n})}{\sigma_{n}^{2}} \right\}  \right\vert \leq\frac{C}%
{\sigma^{2}_{n}},
\]
where $x_{n,k}=\sigma_{n}^{-1}(k-\nu_{n}+1/2)$ for any $k\in\mathbb{Z}$,
$\Phi(x)=(2\pi)^{-1/2}\int_{-\infty}^{x}\exp(-z^{2}/2)dz$ is the cdf of the
standard normal distribution $\mathcal{N}(0,1)$, $\phi(x)=\Phi^{\prime}(x)$
and $C<+\infty$ is a universal constant.
\end{theorem}

Observe first that we can write:
\begin{multline*}
\mathbb{P}\left\{  \mathcal{M}_{N}=0\right\}  =\mathbb{P}\left\{\sum_{i=1}^{N}(\epsilon
_{i}-p_{i})\in]-1/2,1/2]\right\}\\
=\mathbb{P}\left\{  d_{N}^{-1/2}\sum_{i=1}^{N}%
m_{i}\leq \frac1 2 d_{N}^{-1/2}\right\}  -\mathbb{P}\left\{  d_{N}^{-1/2}%
\sum_{i=1}^{N}m_{i}\leq-\frac1 2 d_{N}^{-1/2}\right\}  .
\end{multline*}
Applying Theorem \ref{thm:denominator} to bound the first term of this
decomposition (with $k=\nu_{n}$ and $x_{n,k}=1/(2\sqrt{d_{N}})$) directly yields
that
\begin{multline*}
 \mathbb{P}\left\{  \frac{\sum_{i=1}^{N}m_{i}}{\sqrt{d_N}}\leq
\frac{1}{2\sqrt{d_N}}\right\}  \geq \Phi\left(\frac{1}{2\sqrt{d_{N}}}\right)\\+\frac{1-\frac{1}{4d_{N}}}%
{6\sqrt{d_{N}}}\phi\left(\frac{1}{2\sqrt{d_{N}}}\right)\left\{  1-\frac{2\sum_{i=1}^{n}p_{i}%
^{2}(1-p_{i})}{d_{N}}\right\}  - \frac{C}{d_{N}}.
\end{multline*}
For the second term, its application with $k=\nu_{n}-1$ entails that:
\begin{multline*}
-\mathbb{P}\left\{  \frac{1}{2\sqrt{d_N}}\sum_{i=1}^{N}m_{i}\geq
-\frac{1}{2\sqrt{d_N}} \right\}  \leq -\Phi\left(-\frac{1}{2\sqrt{d_{N}}}\right)\\-\frac{1-\frac{1}{4d_{N}}}{6\sqrt{d_{N}}}\phi\left(-\frac{1}{2\sqrt{d_{N}}}\right) 
\left\{  1-\frac{2\sum_{i=1}^{n}p_{i}%
^{2}(1-p_{i})}{d_{N}}\right\} -\frac{C}{d_{N}}.
\end{multline*}
If $d_N\geq 1$, it follows that
\begin{eqnarray*}
\mathbb{P}\left\{  \mathcal{M}_{N}=0\right\}  &\geq &\Phi\left(\frac{1}{2\sqrt{d_{N}}}%
\right)-\Phi\left(-\frac{1}{2\sqrt{d_{N}}}\right)- \frac{2C}{d_{N}}\\
&= & 2\int_{0}^{\frac{1}{2\sqrt{d_N}}}\phi(t)dt- \frac{2C}{d_{N}}
\geq \left( \phi(1/2)-\frac{2C}{\sqrt{d_N}}\right)\frac{1}{\sqrt{d_N}}.
\end{eqnarray*}
 We
thus obtain the desired result for $d_{N}\geq D$, where $D>0$ is any constant strictly larger than $4C^2\phi^2(1/2)$.

\subsection{Proof of Lemma \ref{lem:numerator}}

Observe that
\begin{equation}
Var \left(  \sum_{i=1}^{N} Z_{i}\right)  =\sum_{i=1}^{N}Var\left(
Z_{i}\right)  =Var \left(  \sum_{i=1}^{N} \epsilon_{i}^{*} \frac{x_{i}}{p_{i}%
}\right)  =Var \left(  \widehat{S}^{\boldsymbol{\epsilon} ^{*}_{N}%
}_{\boldsymbol{p} _{N}} \right)  .
\end{equation}
Let $\psi_{N}(u)=\log\mathbb{E}^{*}[\exp(\langle u,(\mathcal{Z}_{N}%
,\;\mathcal{M}_{N})\rangle)]$, $u=(u_{1},u_{2})\in\mathbb{R}^{+}%
\times\mathbb{R}$, be the log-Laplace of the $1$-lattice random vector
$(\mathcal{Z}_{N},\;\mathcal{M}_{N})$, where $\langle.,\; .\rangle$ is the
usual scalar product on $\mathbb{R}^{2}$. Denote by $\psi_{N}^{(1)}(u)$ and
$\psi_{N}^{(2)}(u)$ its gradient and its Hessian matrix respectively. Consider
now the conditional probability measure $\mathbb{P}^{*}_{u,N}$ given
$\mathcal{D}_{N}$ defined by the Esscher transform
\begin{equation}
d\mathbb{P}_{u,N}=\exp\left(  \left\langle u,(\mathcal{Z}_{N},\mathcal{M}%
_{N})\right\rangle -\psi_{N}(u)\right)  d\mathbb{P}.
\end{equation}
The $\mathbb{P}_{u,N}$-expectation is denoted by $\mathbb{E}_{u,N}[.]$, the
covariance matrix of a $\mathbb{P}_{u,N}$-square integrable random vector $Y$
under $\mathbb{P}_{u,N}$ by $Var_{u^{*},N}(Y)$. With $x=t\sqrt{n}/N$, by
exponential change of probability measure, we can rewrite the numerator of
\eqref{eq:ratio}  as
\begin{multline}
\mathbb{P}\left\{  \mathcal{Z}_{N}\geq x,\mathcal{M}_{N}=0\right\}
=\mathbb{E}_{u,N}\left[  e^{\psi_{N}(u)-\left\langle u,(\mathcal{Z}%
_{N},\mathcal{M}_{N})\right\rangle }\mathbb{I}\{\mathcal{Z}_{N}\geq
x,\mathcal{M}_{N}=0\}\right] \nonumber\\
=H(u)\mathbb{E}_{u,N}\left[  e^{-\left\langle u,(\mathcal{Z}_{N}%
-x,\mathcal{M}_{N})\right\rangle }\mathbb{I}\{\mathcal{Z}_{N}\geq
x,\mathcal{M}_{N}=0\}\right]  ,
\end{multline}
where we set $H(u)=\exp(-\left\langle u,(x,0)\right\rangle +\psi_{N}(u))$.
Now, as $\psi_{N}$ is convex, the point defined by
\[
u^{\ast}=(u_{1}^{*},0)=\arg\sup_{u\in\mathbb{R}_{+}\times\{0\}}\{\langle
u,(x,0)\rangle-\psi_{N}(u)\}
\]
is such that $\psi_{N}^{(1)}(u^{\ast})=(x,0)$. Since $\mathbb{E}%
[\exp(<u,(\mathcal{Z}_{N},\mathcal{M}_{N})>)]=\exp(\psi_{N}(u))$, by
differentiating one gets
\[
\mathbb{E}[e^{<u^{\ast},(\mathcal{S}_{N},\;\mathcal{M}_{N})>}(\mathcal{S}%
_{N},\;\mathcal{M}_{N})]=\psi_{N}^{(1)}(u^{\ast})e^{\psi_{N}(u^{\ast}%
)}=(x,0)e^{\psi_{N}(u^{\ast})},
\]
so that $\mathbb{E}_{u^{\ast},N}[(\mathcal{Z}_{N},\mathcal{M}_{N})]=(x,0)$ and
$Var_{u^{\ast},N}[(\mathcal{Z}_{N},\mathcal{M}_{N})]=\psi_{N}^{(2)}(u^{\ast}%
)$. Choosing $u=u^{*}$, integration by parts combined with straightforward
changes of variables yields
\[
\mathbb{E}_{u^{*},N}[e^{-\left\langle u^{*},(\mathcal{Z}_{N}-x,\mathcal{M}%
_{N})\right\rangle }\mathbb{I}\{\mathcal{Z}_{N}\geq x,\mathcal{M}%
_{N}=0\}]\newline \leq\mathbb{P}^{*}_{u^{*},N}\left\{  \mathcal{M}%
_{N}=0\right\}  .
\]
Hence, we have the bound:
\begin{equation}
\label{eq:prod}\mathbb{P}\left\{  \mathcal{Z}_{N}\geq x,\mathcal{M}%
_{N}=0\right\}  \leq H(u^{*})\times\mathbb{P}_{u^{*},N}\left\{  \mathcal{M}%
_{N}=0\right\}  .
\end{equation}
We shall bound each factor involved in \eqref{eq:prod}  separately. We start
with bounding $H(u^{*})$, which essentially boils down to bounding
$\mathbb{E}[e^{\langle u^{*},(\mathcal{Z}_{N},\mathcal{M}_{N})\rangle}]$.

\begin{lemma}
\label{lem:factor1} Under Theorem \ref{thm:rejective}'s assumptions, we have:
\begin{align}
H(u^{*})  &  \leq\exp\left(  -\frac{Var\left(  \sum_{i=1}^{N}Z_{i} \right)
}{\left(  \max_{1\leq j\leq N}\vert x_{j}\vert/p_{j} \right)  ^{2}}h\left(
\frac{N}{\sqrt{n}}\frac{x\max_{1\leq j\leq N}\vert x_{j}\vert/p_{j}%
}{Var\left(  \sum_{i=1}^{N}Z_{i} \right)  } \right)  \right) \\
&  \leq\exp\left(  -\frac{N^{2}x^{2}/n}{2\left(  Var\left(  \sum_{i=1}%
^{N}Z_{i} \right)  +\frac{1}{3}\frac{N}{\sqrt{n}}x \max_{1\leq j\leq N}\vert
x_{j}\vert/p_{j} \right)  } \right)  ,
\end{align}
where $h(x)=(1+x)\log(1+x)-x$ for $x\geq0$.
\end{lemma}

\begin{proof}
Using the standard argument leading to the Bennett-Bernstein bound, observe that: $\forall i\in \{1,\; \ldots,\; N  \}$, $\forall u_1> 0$,
\begin{equation*}
\mathbb{E}[e^{u_1Z_i}%
]\leq \exp\left( Var(Z_i)\frac{\exp\left(u_1\max_{1\leq j\leq N}%
\frac{\vert x_j\vert}{p_j} \right) - 1- u_1\max_{1\leq j\leq N}%
\frac{\vert x_j\vert}{p_j}}{\left(\max_{1\leq j\leq N}\frac{\vert x_j\vert}%
{p_j}\right)^2} \right).
\end{equation*}
since we $\mathbb{P}%
$-almost surely have  $\vert Z_{i}\vert \leq \max_{1\leq j\leq N}%
\vert x_j\vert/p_j $ for all $i\in \{1,\; \ldots,\; N  \}$. Using the independence of the $Z_i$'s, we obtain that: $\forall u_1> 0$,
\begin{multline*}
\mathbb{E}[e^{u_1 \mathcal{Z}_N}]\leq \\
\exp\left( Var\left(  \sum_{i=1}%
^NZ_i\right)\frac{\exp\left(\frac{\sqrt{n}}{N}u_1\max_{1\leq j\leq N}%
\frac{\vert x_j\vert}{p_j} \right) - 1- \frac{\sqrt{n}}{N}%
u_1\max_{1\leq j\leq N}\frac{\vert x_j\vert}{p_j}}{\left(\max_{1\leq j\leq N}%
\frac{\vert x_j\vert}{p_j}\right)^2} \right).
\end{multline*}%

The resulting upper bound for $H((u_1,0))$ being minimum for
$$
u_1=\frac{N}%
{\sqrt{n}}\frac{\log \left( 1+\frac{N}{\sqrt{n}}\frac{x\max_{1\leq j\leq N}%
\vert x_j\vert/p_j}{Var(\sum_{i=1}^NZ_i)} \right)}{\max_{1\leq j\leq N}%
\vert x_j\vert/p_j},
$$
this yields
\begin{equation}
H(u^*)\leq \exp\left( -\frac{Var\left( \sum_{i=1}^NZ_i \right)}%
{\left(\max_{1\leq j\leq N}\vert x_j\vert/p_j \right)^2}h\left( \frac{N}%
{\sqrt{n}}\frac{x\max_{1\leq j\leq N}\vert x_j\vert/p_j}{Var\left( \sum_{i=1}%
^NZ_i \right)} \right) \right).
\end{equation}%

Using the classical
inequality
\begin{equation*}
h(x)\geq \frac{x^{2}%
}{2(1+x/3)},\text{ for }x\geq 0,
\end{equation*}%

we also get that
\begin{equation}\label{eq:prod}
H(u^*)\leq \exp\left( -\frac{N^2x^2/n}{2\left(Var\left( \sum_{i=1}%
^NZ_i \right)+\frac{1}{3}\frac{N}{\sqrt{n}}x \max_{1\leq j\leq N}%
\vert x_j\vert/p_j \right)} \right).
\end{equation}
$\square$
\end{proof}


We now prove the lemma stated below, which provides an upper bound for
$\mathbb{P}^{*}_{u^{*},N}\{ \mathcal{M}_{N}=0\}$.

\begin{lemma}
\label{lem:factor2} Under Theorem \ref{thm:rejective}'s assumptions, there
exists a universal constant $C^{\prime}$ such that: $\forall N\geq1$,
\begin{equation}
\mathbb{P}_{u^{*},N}\left\{  \mathcal{M}_{N}=0\right\}  \leq C^{\prime
}\frac{1}{\sqrt{d_{N}}}.
\end{equation}
\end{lemma}

\begin{proof}
Under the probability measure $\mathbb{P}_{u^*,N}%
$, the $\varepsilon _{i}%
$'s are still
independent Bernoulli variables, with means now given by
\begin{equation*}
\pi^* _{i}\overset{def}{=}\sum_{s\in
\mathcal{P}(\mathcal{I}_{N}%
)}e^{\left\langle u^*,(\mathcal{Z}_{N}(s),%
\mathcal{M}_{N}%
(s))\right\rangle -\psi _{N}(u^*)}\mathbb{I}\left\{ i\in
s\right\} R_{N}%
(s)>0,
\end{equation*}
for $i\in \{1,\; \ldots,\; N  \}$.
Since $\mathbb{E}%
_{u^{\ast },N}[\mathcal{M}_{N}]=0$, we have $\sum_{i=1}^{N}\pi^*_{i}%
=n$ and thus
\begin{equation*}
d_{N,u^{\ast
}}\overset{def}{=}%
Var_{u^{\ast },N}\left(\sum_{i=1}^{N}\varepsilon _{i}\right)=\sum_{i=1}%
^{N}\pi^* _{i}(1-\pi ^*_{i})\leq n.
\end{equation*}%

We can thus apply the local Berry-Esseen bound established in \cite{Deheuvels}
for sums of independent (and possibly non identically) Bernoulli random variables, recalled in Theorem \ref{thm:numerator}%
.
\begin{theorem}\label{thm:numerator}(\cite{Deheuvels}%
, Theorem 1.2) Let $(Y_{j,n})_{1\leq j\leq n}%
$ be a triangular array of  independent Bernoulli random variables with means $q_{1,n}%
,\; \ldots,\; q_{n,n}%
$ in $(0,1)$ respectively. Denote by $\sigma^2_n=\sum_{i=1}^nq_{i,n}%
(1-q_{i,n})$ the variance of the sum $\Sigma_n=\sum_{i=1}^nY_{i,n}%
$ and by $\nu_n=\sum_{i=1}^nq_{i,n}%
$ its mean. Considering the cumulative distribution function (cdf) $F_n(x)=\mathbb{P}%
\{ \sigma_n^{-1}%
(\Sigma_n-\nu_n )\leq x \}$, we have: $\forall n\geq 1$,
\begin{equation}
\sup_x\left(1+\vert x\vert^3  \right)\left\vert  F_n(x)-\Phi(x)\right\vert\leq \frac{C}%
{\sigma_n},
\end{equation}
where $\Phi(x)=(2\pi)^{-1/2}\int_{-\infty}%
^x\exp(-z^2/2)dz$ is the cdf of the standard normal distribution $\mathcal{N}%
(0,1)$ and $C<+\infty$ is a universal constant.
\end{theorem}%

Applying twice a pointwise version of the bound recalled above (for $x=0$ and $x=1/\sqrt{d_{N,u^*}%
}$), we obtain that
\begin{multline*}
\mathbb{P}_{u^*,N}\left\{ \mathcal{M}%
_{N}=0\right\}= \mathbb{P}_{u^*,N}\left\{d_{N,u^{\ast
}}^{-1/2}
\sum_{i=1}^Nm_i\leq 0\right\}\\-  \mathbb{P}_{u^*,N}\left\{d_{N,u^{\ast
}%
}^{-1/2} \sum_{i=1}^Nm_i\leq -d_{N,u^{\ast
}}^{-1/2}%
\right\}\\
\leq \frac{2C}{\sqrt{d_{N,u^*}}}+\Phi(0)-\Phi(-d_{N,u^*}%
^{-1/2})\leq  \left(\frac{1}{\sqrt{2\pi}}+2C\right)\frac{1}{\sqrt{d_{N,u^*}}%
},
\end{multline*}%

by means of the finite increment theorem.
Finally, observe that:
\begin{multline*}
d_{N,u^*}%
=\mathbb{E}_{u^*, N}\left[\left(\sum_{i=1}^Nm_i\right)^2 \right]=\mathbb{E}%
\left[ \left(\sum_{i=1}^Nm_i\right)^2/H(u^*) \right]\\
\geq \mathbb{E}%
\left[ \left(\sum_{i=1}^Nm_i\right)^2 \right]=d_N,
\end{multline*}%

since we proved that $H(u^*)\leq 1$. Combined with the previous bound, this yields the desired result.
$\square$
\end{proof}%

Lemmas \ref{lem:factor1} and \ref{lem:factor2} combined with Eq.
\eqref{eq:prod}  leads to the bound stated in Lemma \ref{lem:numerator}.

\subsection*{Proof of Theorem \ref{thm:final}}
We start with proving the preliminary result below.

\begin{lemma}
\label{lem:bias}Let $\pi_{1},\; \ldots, \pi_N$ be the first order
inclusion probabilities of a rejective sampling of size $n$ with canonical representation characterized by the Poisson weights $p_1,\; \ldots,\; p_N$.Provided that $d_{N}=\sum_{i=1}^{N}p_{i}(1-p_{i}%
)\geq1$, we have: $\forall i\in\{1,\; \ldots,\; N  \}$,
\[
\left\vert \frac{1}{\pi_{i}}-\frac{1}{p_{i}}\right\vert\leq\frac{6}{d_{N}}\times \frac{1-\pi_{i}}%
{\pi_{i}}.
\]
\end{lemma}
\begin{proof}
The proof follows the representation (5.14) on page 1509 of \cite{Hajek64}.
We have
 For all $i\in \{1,\; \ldots,\; N\}$, we have:
\begin{eqnarray*}
\frac{\pi_{i}}{p_{i}}\frac{1-p_{i}}{1-\pi_{i}}  &=&\frac{\sum_{s\in \mathcal{P}(\mathcal{I}_N):\;  i\in \mathcal{I}_N\setminus\{s \}}P(s)\sum_{h\in s}\frac{1-p_{h}}%
{\sum_{j\in s}(1-p_{j})+(p_{h}-p_{i})}}{ \sum_{s\in \mathcal{P}(\mathcal{I}_N):\; i\in
\mathcal{I}_N\setminus\{s \}}P(s)}\\
&=&\frac{\sum_{s:\ i\in \mathcal{I}_N\setminus\{s \}}%
P_N(s)\sum_{h\in s}\frac{1-p_{h}}{\sum_{j\in s}(1-p_{j})\left(  1+\frac
{(p_{h}-p_{i})}{\sum_{j\in s}(1-p_{j})}\right)  }}{\sum_{s:\ i\in \mathcal{I}_N\setminus\{s \}}P_N(s)}.
\end{eqnarray*}
Now recall that for any $x\in]-1,1[ $, we have:
\[
1-x\leq\frac{1}{1+x}\leq1-x+x^{2}.
\]
It follows that
\begin{align*}
\frac{\pi_{i}}{p_{i}}\frac{1-p_{i}}{1-\pi_{i}}  & \leq1-\left(  \sum_{s:\ i\in
\mathcal{I}_N\setminus\{s \}}P(s)\right)  ^{-1}\sum_{s:\ i\in \mathcal{I}_N\setminus\{s \}}P(s)\sum_{h\in s}\frac{(1-p_{h}%
)(p_{h}-p_{i})}{\left(  \sum_{j\in s}(1-p_{j})\right)  ^{2}}\\
& +\left(  \sum_{s:\ i\in \mathcal{I}_N\setminus\{s \}}P(s)\right)  ^{-1}\sum_{s:\ i\in \mathcal{I}_N\setminus\{s \}}%
P(s)\sum_{h\in s}\frac{(1-p_{h})(p_{h}-p_{i})^{2}}{\left(  \sum_{j\in
s}(1-p_{j})\right)  ^{3}}%
\end{align*}
Following now line by line the proof on p. 1510 in \cite{Hajek64} and noticing that $\sum_{j\in s}%
(1-p_{j})\geq1/2d_{N}$ (see Lemma 2.2 in \cite{Hajek64}), we have
\begin{align*}
\left\vert \sum_{h\in s}\frac{(1-p_{h})(p_{h}-p_{i})}{\left(  \sum_{j\in
s}(1-p_{j})\right)  ^{2}}\right\vert  & \leq\frac{1}{\left(  \sum_{j\in
s}(1-p_{j})\right)  } \leq\frac{2}{d_{N}}
\end{align*}
and similarly%
\begin{align*}
\sum_{h\in s}\frac{(1-p_{h})(p_{h}-p_{i})^{2}}{\left(  \sum_{j\in s}%
(1-p_{j})\right)  ^{3}}  & \leq\frac{1}{\left(  \sum_{j\in s}(1-p_{j})\right)
^{2}} \leq\frac{4}{d_{N}^{2}}.
\end{align*}
This yieds: $\forall i\in\{1,\; \ldots,\; N  \}$,
\[
1-\frac{2}{d_{N}}\leq\frac{\pi_{i}}{p_{i}}\frac{1-p_{i}}{1-\pi_{i}}\leq
1+\frac{2}{d_{N}}+\frac{4}{d_{N}^{2}}%
\]
and
\[
p_{i}(1-\pi_{i})(1-\frac{2}{d_{N}})\leq\pi_{i}(1-p_{i})\leq p_{i}(1-\pi
_{i})\left(1+\frac{2}{d_{N}}+\frac{4}{d_{N}^{2}}\right),
\]
leading then to
\[
-\frac{2}{d_{N}}(1-\pi_{i})p_{i}\leq\pi_{i}-p_{i}\leq p_{i}\left(1-\pi_{i}%
\right)\left(\frac{2}{d_{N}}+\frac{4}{d_{N}^{2}}\right)
\]
and finally to
\[
-\frac{(1-\pi_{i})}{\pi_{i}}\frac{2}{d_{N}}\leq\frac{1}{p_{i}}-\frac{1}%
{\pi_{i}}\leq\frac{(1-\pi_{i})}{\pi_{i}}\left(\frac{2}{d_{N}}+\frac{4}{d_{N}^{2}}\right).
\]
Since $1/d_{N}^{2}\leq 1/d_{N}$ as soon as $d_{N}\geq1$, the lemma is proved.  $\square$
\end{proof}
\medskip

By virtue of lemma \ref{lem:bias}, we obtain that:
\begin{equation*}
\left\vert \widehat{S}_{\boldsymbol{\pi}_N}^{\boldsymbol{\epsilon}_N^*}-\widehat{S}_{\boldsymbol{p} _{N}%
}^{\boldsymbol{\epsilon} _{N}^{\ast}}\right\vert
 \leq\frac{6}{d_{N}}\sum_{i=1}^{N}\frac{1}{\pi_{i}}|x_{i}|=M_{N}%
\end{equation*}
It follows that
\begin{align*}
\mathbb{P}\left\{  \widehat{S}_{\boldsymbol{\pi}_N}^{\boldsymbol{\epsilon}_N^*}-S_N>x\right\}    &
\leq\mathbb{P}\left\{  |\widehat{S}_{\boldsymbol{\pi}_N}^{\boldsymbol{\epsilon}_N^*}-\widehat{S}_{\boldsymbol{p} _{N}%
}^{\boldsymbol{\epsilon} _{N}^{\ast}}|+\widehat
{S}_{\boldsymbol{p} _{N}}^{\boldsymbol{\epsilon} _{N}^{\ast}}-S_{N}>x\right\}
\\
& \leq\mathbb{P}\left\{  M_{N}+\widehat{S}_{\boldsymbol{p} _{N}}%
^{\boldsymbol{\epsilon} _{N}^{\ast}}-S_{N}>x\right\}
\end{align*}
and a direct application of Theorem \ref{thm:rejective} finally gives the desired result.

\bibliographystyle{plain}
\bibliography{ERM_sampling}
\end{document}